\documentclass{article}
\usepackage[T1]{fontenc}
\usepackage[utf8]{inputenc}
\usepackage[icelandic, english]{babel}
\usepackage{amsmath}
\usepackage{amssymb}
\usepackage{amsthm}
\usepackage{parskip}
\usepackage{hyperref}
\usepackage{pgfplots}
\title{The difficulty of beating the Taxman}
\author{
Atli Fannar Franklín, \textit{The University of Iceland} \\
Robert K. Moniot, \textit{Fordham University}
}
\date{\today}

\pgfplotsset{compat=1.17}

\newcommand\floor[1]{\left\lfloor#1\right\rfloor}
\newcommand\p[1]{\left(#1\right)}
\newcommand\cp[1]{\left\{#1\right\}}

\newtheorem{theorem}{Theorem}

\newtheorem{definition}{Definition}

\begin{document}

\maketitle

\section*{Introduction}

Taxman was invented around the year 1970 by Diane Resek of San Francisco State University while she worked at the Lawrence Hall of Science in Berkeley \cite{diane}. It was made as a teaching tool, providing a more engaging method of practising arithmetic. The game soon became popular with teachers of computer science as a programming exercise, since it is fairly easy but not trivial to implement, and provides a gentle introduction to important algorithm design principles \cite{example}. The game is sometimes referred to as Number Shark or Zahlenhai. \\

The Taxman game is an adversarial game played against the titular Taxman. The Taxman's moves are fully deterministic, so it is a one person game. The game starts with all the positive integers from $1$ to some maximum $N$ in play. The player's moves consists of choosing a number in play and adding it to their score, removing it from play after. The Taxman then takes all its divisors and adds them to their own score, the tax. The player is not allowed to take a number that results in no tax. Then at the end the Taxman gets any remaining numbers. The victor is the one with the greater score. \\

The game has been studied to find optimal sequences of picks \cite{chess, hoey}.
The optimal scores as a function of pot size form a sequence that is listed
on the Online Encyclopedia of Integer Sequences (OEIS) \cite{oeis}, sequence
A019312. Since finding optimal play appears likely to be
NP-hard, efforts have been made to find heuristic strategies that do well
in practice \cite{mainref, opt2}. The existence of winning strategies has been proven, albeit up to now only for values of $N$ larger than some undetermined and
quite large value \cite{opt1,perl}. These efforts approached the problem from a number theory perspective.  The present work introduces a graph theoretic view that leads to a more tractable formulation. \\

In this paper we present an equivalence between valid sequences of moves in the Taxman game and certain graph theoretic constraints. This is then used to show that a generalized version of Taxman is NP-hard. After this we present a heuristic method that provides a winning move sequence for the original Taxman game for all $N > 3$ in $\mathcal{O}(N\log(N))$ time. Lastly we present two algorithms that produce good lower and upper bounds on the optimal score, both running in $\mathcal{O}(N^2\log(N))$. 

\section*{Generalizing the Taxman}

Before defining a more general notion of taxman we will need to define partially ordered sets. \\

\begin{definition}
A strict partial order is a set $P$ along with a relation $<$ satisfying the following three properties. \\
\begin{itemize}
\item No $p \in P$ satisfies $p < p$ (irreflexivity).
\item If $a, b, c \in P$, $a < b$ and $b < c$ then $a < c$ (transitivity).
\item If $a, b \in P$ and $a < b$ then $b < a$ does not hold (asymmetry). \\
\end{itemize}
For such a poset we will let $q \leq p$ denote the fact that $q < p$ or $q = p$. \\
\end{definition}

To simplify notation going forward we will also define some additional notation before moving on. \\

\begin{definition}
Let $(P, <)$ be a strict partial order. For $p, q \in P$ we say that $p$ covers $q$, denoted $q \lessdot p$, if $q < p$ and there exists no $x \in P$ such that $q < x < p$. \\
\end{definition}

With this in mind we can give the following definition of the generalized taxman game. \\

\begin{definition}
Let us have some finite strict partial order $(P, <)$ and a weight function $w : P \rightarrow \mathbb{R}$. We define the generalized taxman game on $(P, <, w)$ as follows. We start with all the elements of $P$ in play. In each move we may pick an item $p \in P$ if some $q \in P$ such that $q < p$ is still left. Then we gain $w(p)$ points, remove $p$ and the taxman removes all $a \in P$ such that $a < p$. Once we run out of valid picks the taxman claims the rest. \\
\end{definition}

We see that by picking $P = \{1, 2, \dots, n\}$, $<$ as the strict divisibility relation and $w$ as the identity function we recover the original game. However, while we can define the taxman game on a general poset, our equivalence will consider a specific kind of poset. Let us thus give one last definition before moving on. \\

\begin{definition}
A graded poset is a poset $(P, <)$ equipped with a rank function $\rho : P \rightarrow \mathbb{N}$ satisfying the two following conditions: \\
\begin{itemize}
\item If $p, q \in P$ and $q < p$ then $\rho(q) < \rho(p)$.
\item If $p, q \in P$ and $q \lessdot p$ then $\rho(q) + 1 = \rho(p)$. \\
\end{itemize}
\end{definition}

We note that the original taxman game is played on a ranked poset. In that case we can simply rank the numbers $1, 2, \dots, n$ by their number of prime factors, counted with multiplicity.

\section*{An equivalent problem}

Before stating the theorem, we give one final definition. \\

\begin{definition}
Let us have some graph $G$ where the vertex set of $G$ is a finite graded poset $(P, <, \rho)$. A matching on $G$ is a subset of the edges of $G$ such that no two edges share any endpoints. A cycle is called alternating if exactly every other edge in the cycle is in the matching. Lastly we will call such a cycle flat if the vertices in the cycle all have rank $n$ or $n + 1$ for some number $n$. \\
\end{definition}

The fact that allows us to relate this form of taxman to a NP-complete problem is the following result. \\

\begin{theorem}
Consider the generalized taxman game on $(P, <, \rho, w)$ where $P$ is a finite graded poset. Construct a graph $G$ with vertex set $P$ and an edge between $x$ and $y$ iff $x \lessdot y$ or $y \lessdot x$. For an edge where $x \lessdot y$ we put the weight $w(y)$ on the edge. We note that this is well defined since $x \lessdot y$ and $y \lessdot x$ can not hold simultaneously. Then the optimal sequence of plays in the generalized taxman game corresponds to a maximum weight matching on $G$ that does not contain any flat alternating cycles.
\end{theorem}

\begin{proof}
We will prove this by demonstrating a bijection between flat-alternating-cycle-free matchings on $G$ and move sequences in the taxman game such that the weight of the matching is equal to the score for the move sequence. Thus if this holds, maximizing one means maximizing the other. \\

We start with the direction of showing that a move sequence for the taxman game will give us an flat-alternating-cycle-free matching on $G$. Suppose we have some optimal sequence of plays $p_1, \dots, p_n$ where $p_i$ denotes the value removed in the $i$-th move. Then by the definition of the taxman game some smaller value or values are removed in each of those moves. Let $q_i$ then be the largest value the taxman gets in the $i$-th move, then $q_i < p_i$. Suppose then $q_i \lessdot p_i$ does not hold for some $i$. Then there must exist an $x$ such that $q_i < x < p_i$. But that means $x$ has been removed at some point. If we chose $x$ at some point $q_i$ wouldn't be an option as well. But if we removed some element $y$ such that $x < y$ then we must have $q_i < x$, so it would have been removed in that case as well. Thus we get a contradiction so $q_i \lessdot p_i$. Thus all of our pairs $(q_i, p_i)$ correspond to edges in our graph $G$. Furthermore our score for choosing the $p_i$ correspond exactly to the weights of the edges. Lastly we must show this is an alternating cycle free matching. \\

We start by showing it is a matching, which means the $p_i$ and $q_i$ are all pairwise distinct. The $p_i$ are internally pairwise distinct as per their definition. We can't have $q_i = q_j$ for $i \neq j$ either since $q_i$ is removed in the $i$-th move but $q_j$ is removed in the $j$-th move and $i \neq j$. For the same reason we can't have $p_i = q_j$ either for $i \neq j$. Thus this is a matching, so let us show it is flat-alternating-cycle-free next. Suppose we have some flat alternating cycle $x_1, x_2, \dots, x_{2n}$. Let $n$ be such that all the $x_i$ are of rank $n$ or $n + 1$. We can shift the indices of the cycles as we like, so WLOG $x_1$ is of rank $n$. Similarly we can reverse the cycle as we like so WLOG there is an edge from $x_1$ to $x_2$. Since the poset is graded, we can't have $x_i$ and $x_{i + 1}$ of the same rank. Thus $x_2$ is of rank $n + 1$, $x_3$ is of rank $n$ and so on. Then we have $x_{2i - 1} \lessdot x_{2i}$ for $i = 1, \dots, n$. Thus the odd indexed values have a lower rank than the even indexed ones. But since this is a cycle in the original graph there must be an edge between $x_{2i}$ and $x_{2i+1}$ as well. Thus the rank tells us that we must have $x_{2i+1} \lessdot x_{2i}$. There are $n$ edges in this cycle that are a part of our matching, so our moveset allows us to obtain $n$ of these values. But $x_{2i-1} \lessdot x_{2i}$ and $x_{2i+1} \lessdot x_{2i}$ so as soon as we make one of the moves, there will be less values left of the lower rank. Thus this could not have been a valid move sequence, giving us our desired contradiction. This completes the first direction of our proof. \\

Now we show the reverse direction. Let us have some flat-alternating-cycle-free matching on $G$. Let us denote the pairs of vertices in our matching with $(x_1, y_1), \dots, (x_n, y_n)$ such that $x_i \lessdot y_i$. We will now show that we can always pick some $y_i$ without invalidating any of the other $y_j$ as legal moves. This means finding a $y_i$ such neither $x_j < y_i$ nor $y_j < y_i$ holds for any $j \neq i$. We note that if $y_i > y_j$ then since $y_j > x_j$ we get $y_i > x_j$ by transitivity. Thus it suffices to show that $y_i > x_j$ does not hold. \\

Thus we now consider a procedure where we start by picking a pair $(x_i, y_i)$ arbitrarily. If it satisfies our desired condition, then we are done. If not, there is some $(x_j, y_j)$ such that $x_j < y_i$. In this case we pick $(x_j, y_j)$ instead. Since our poset is finite we can repeat this procedure until one of two things happens. In the first case we find a pair satisfying our condition, in which case we are done. Otherwise we must at some point encounter a pair we've encountered before. Let us prove the second case can not occur by contradiction. Assume then we have a sequence of pairs $(x_{i_1}, y_{i_1}), \dots, (x_{i_m}, y_{i_m})$ such that $x_{i_1} < y_{i_2}$, $x_{i_2} < y_{i_3}$ and so on in addition to $x_{i_m} < y_{i_1}$. Then $\rho(y_{i_j}) = \rho(x_{i_j}) + 1$ and $\rho(x_{i_j}) < \rho(y_{i_{j+1}})$. Combining these we get $\rho(y_{i_j}) \leq \rho(y_{i_{j+1}})$. But this holds cyclically, so going around the entire cycle the values are squeezed together. Thus all the $y_{i_j}$ have the same rank and all the $x_{i_j}$ have the same rank, one lower than that of the $y_{i_j}$. Thus no value can fit in between $x_{i_j}$ and $y_{i_{j+1}}$ so we get $x_{i_j} \lessdot y_{i_{j+1}}$. But now this is a cycle in the original graph with exactly every other edge in the matching. And furthermore the vertices are contained in two adjacent ranks, so it is flat. This contradicts the fact that our matching is flat-alternating-cycle-free, so this can not occur.
\end{proof}

\begin{theorem}
Solving the generalized taxman game optimally is NP-hard. More specifically it is NP-hard for the case when the weight function is the constant function $1$.
\end{theorem}

\begin{proof}
Let us show that the poset for the generalized taxman game can be chosen such that the graph $G$ becomes any bipartite graph. Let us then have some bipartite graph with halves $A, B$. Let us consider the edges of this graph to be oriented from $A$ to $B$. We let our poset be $P = A \cup B$, defining $q < p$ if there is an edge from $p$ to $q$. We let the rank function take the value $0$ on $B$ and value $1$ on $A$. Going through the definitions above we quickly see that this will be a valid ranked finite poset. Clearly this will also produce exactly our desired graph in the theorem above. Furthermore any alternating cycle in this graph will be contained in two ranks. We also restrict ourselves to the unit weight function, so the maximum weight matching is simply the maximum cardinality matching. Thus we see that a polynomial time solution to the generalized taxman game would give a polynomial time solution to finding the maximum cardinality alternating cycle free matching in an arbitrary bipartite graph. By \cite{npresult} this is an NP-complete problem.
\end{proof}

\section*{A winning strategy for the original Taxman game}

We now present an efficient algorithm for solving the Taxman problem that is
non-optimal but capable of winning the game for all $N > 3$. We start by constructing some sets of pairs from the values $1$ to $N$. Let us define

\[S_p = \cp{(x, px) | \, x \in \mathbb{N},  px \leq N }\]

Our algorithm runs through every prime $p \leq N$ in descending order. For each such prime $p$ it runs through the pairs in $S_p$ in descending order and picks every pair that does not share any endpoints with earlier picks. These chosen pairs will then form the matching corresponding to the solution. Thus we need to prove that this forms a matching and does not contain any flat alternating cycles. \\

\begin{theorem}
    Our given algorithm produces a flat alternating cycle-free matching.
\end{theorem}

\begin{proof}
	
    To show that this is a matching, we only have to show that no two pairs in $S_p$ have any end points in common. But this is clear from the definition of the algorithm. Thus we only have to show that it contains no flat alternating cycles. Suppose we have some flat alternating cycle $x_1, y_1, x_2, y_2, \dots, x_r, y_r$. Without loss of generality we can choose the naming such that $(x_i, y_i)$ are the pairs in our matching and $(y_i, x_{i+1})$ are the ones not in the matching, indices considered modulo $r$. Furthermore we can choose the names such that $y_i > x_i$. As we walk through these numbers in order we only change one prime factor in each step. Let $p$ be the largest such prime factor that occurs in the cycle. Suppose then $i$ is such that $y_i = px_i$. Then we need to remove that factor of $p$ at some point to end up where we started. Thus for some $j$ we have $y_j = p x_{j + 1}$, so since $x_j \neq x_{j+1}$ we have $y_j \neq p x_j$ but $(x_j, y_j)$ is in the matching. Since $p$ was the largest such prime factor in the cycle, we must have some prime $q < p$ such that $y_j = q x_j$. Similarly there is some prime $r < p$ such that $y_{j+1} = r x_{j+1}$. Thus at the point in the algorithm when we considered $S_q$ both $x_j$ and $y_j$ were free to be taken. Similarly we have that $x_{j+1}, y_{j+1}$ were free when we considered $S_r$. But this means they would also have been free to be taken when we considered $S_p$ since $p > q, r$. But if this were the case $(x_{j+1}, y_j)$ would have been in the matching which gives us a contradiction. Thus there are no flat alternating cycles.
    
\end{proof}

Thus this algorithm yields a set of numbers that can be put into an order
corresponding to a valid Taxman game. We address the problem of finding
the ordering of the picks in Theorem 6 below. Since this matching is 
automatically free of flat alternating cycles, we will call it the ``born-free''
matching, and the resulting algorithm for playing Taxman we will call the born-free
matching algorithm. We have not yet proved that it is a winning strategy, i.e.,
that it gets more than half of the sum of the integers in the pot for large enough
$N$. It will be easier to prove this for a modified algorithm that is 
the same as above except that for all $N$, the
only primes used are less than or equal to $5$. Clearly the original algorithm
does at least as well as this modified algorithm. 
We will call this modified algorithm the ``born-free matching with $p_{max} =
5$'' algorithm, to distinguish it from the original one. \\

\begin{theorem}
    For $N \geq 847$ the born-free matching with $p_{max}=5$ algorithm will take more than half the pot.
\end{theorem}

\begin{proof}
We only need to bound it from below, so we can omit terms as desired. We start by bounding the sum obtained by the pairs in $S_5$. It will match every value in $]N/25, N/5]$ to its multiple of $5$. We can then use the bound $N/d \geq \floor{N/d} \geq (N - d + 1) / d$ to get
    \begin{small}
    \[\sum_{i = \floor{N/25} + 1}^{\floor{N/5}} 5i = 5\frac{\floor{\frac{N}{5}}\p{\floor{\frac{N}{5}} + 1}}{2} - 5\frac{\floor{\frac{N}{25}}\p{\floor{\frac{N}{25}} + 1}}{2} \geq \frac{12}{125}N^2 - \frac{2}{5}N - \frac{2}{5}\]
    \end{small}
Next we consider pairs from $S_3$. It will match every value in $]N/5,N/3]$ to its multiple of $3$ as long as it's not a multiple of $5$. The sum of all multiples of $3$ in the interval is given by
    \[\sum_{i = \floor{N/5}+1}^{\floor{N/3}} 3i\]
We need to subtract the multiples of five from this. If $i = 5j$ for some $j$ then $N/5 < i \leq N/3$ translates to $N/25 < j \leq N/15$. Thus the subtracted sum becomes
    \[\sum_{j = \floor{N/25}+1}^{\floor{N/15}} 15j\]
Using the bound $N/d \geq \floor{N/d} \geq (N - d + 1)/d$ again we can get that the difference between these two sums is at least
    \begin{align*}
        &\frac{3}{2} \frac{N - 3 + 1}{3} \p{\frac{N - 3 + 1}{3} + 1} - \frac{3}{2} \frac{N}{5} \p{\frac{N}{5} + 1} \\
        &\quad - \frac{15}{2} \frac{N}{15} \p{\frac{N}{15} + 1} + \frac{15}{2} \frac{N - 25 + 1}{25} \p{\frac{N - 25 + 1}{25} + 1}\\
        &= \frac{32}{375} N^2 - \frac{466}{375}N - \frac{233}{375}
    \end{align*}
    Lastly we consider $S_2$. Here every value in $]N/3,N/2]$ is matched to its multiple of $2$ so long as it's neither a multiple of $3$ nor $5$. Using inclusion-exclusion along with the same bounds as before we can get the lower bound:
    \[\frac{2}{27} N^2 - \frac{362}{135}N - \frac{181}{135}\]
    In total we have shown that the matching achieves a ratio of
    \[\frac{\frac{1724}{3375}N^2 - \frac{29188}{3375}N - \frac{15944}{3375}}{N^2+N}\]
    If we calculate out the derivative we can get that it is $\geq 0$ for positive $N$. Furthermore if we solve for when this ratio is equal to $1/2$ we get $N \approx 846.4$, so the ratio will be greater than $1/2$ for all $N \geq 847$.
\end{proof}

The algorithm is simple enough, so using a computer all values below $847$ can be checked using the unmodified algorithm. It only fails to win on $1, 3, 7$ and $13$ when checked against all $N < 847$. For $N = 1$ the game ends immediately and the Taxman wins. For $N = 3$ the optimal move is to take the $3$, giving a tie. For all other values we can then win. For $N = 7$ we can take $7, 4, 6$ and for $N = 13$ take $13, 9, 10, 8, 12$. Thus we have proven: \\

\begin{theorem}
    For all $N \notin \{1, 3\}$ the taxman game can be won. \\
\end{theorem}

This leaves only the issue of efficiently constructing the order the moves should be made in given the set from the algorithm above. Luckily this can be done very efficiently. \\

\begin{theorem}
    Given a flat alternating cycle-free matching for the standard taxman game, the order for the moves can be calculated in $\mathcal{O}(N\log(N))$ time, assuming constant time integer operations.
\end{theorem}

\begin{proof}
We start by using the sieve of Eratosthenes to get the smallest prime factor of every number from $1$ to $N$ in $\mathcal{O}(N\log(\log(N)))$ time, storing the results. From this we can calculate the rank, that is to say number of prime factors counted with multiplicity, of every number from $1$ to $N$ in $\mathcal{O}(N)$ time. Thus we can partition our set of matched numbers by rank, creating a list for each rank and populating them in $\mathcal{O}(N)$. We can then consider each rank independently if we consider them in increasing order, since picking an item can only prevent picks of lower rank in the future. For each rank we construct a bipartite graph on the matched numbers of two consecutive ranks. We wish to place an edge between two vertices if they differ only by a single prime factor. Take some number of the higher rank, its smallest prime factor can be found repeatedly and divided out to get all prime factors in $\mathcal{O}(\log(N))$ using the sieve. Thus testing those one at a time we can construct the graph in $\mathcal{O}(\log(N))$ time per vertex, for a total of $\mathcal{O}(N\log(N))$ over all the graphs. \\

Suppose now that all the higher rank vertices in one of these graphs have degree $\geq 2$. Then let us start at some vertex $v$. We can then repeatedly travel to a vertex of lower rank that's not in the matching since the degree is $\geq 2$ and the matched edge only contributes $1$ to the degree. We can then travel back up the matched edge since we only include vertices that are part of our matching. Since our graph is finite this must produce a cycle, a flat alternating cycle. But by our assumption no such cycle exists. Thus there exists a vertex of higher rank with degree exactly $1$. Thus we maintain a queue of such vertices and repeatedly delete the front element of that queue along with its matched vertex from our bipartite
graph. This reduces the degree of all the higher-rank vertices connected to
the deleted lower-rank vertex by 1. When the degree of a vertex reaches
1, it is pushed onto the queue. Using appropriate data structures this can be done in $\mathcal{O}(\log(N))$ time per vertex, giving a total of $\mathcal{O}(N\log(N))$. Our order for this one rank is then simply the order in which we deleted the matches from the graph. This thus produces an order in $\mathcal{O}(N\log(N))$ time.
\end{proof}

\section*{Lower and upper bound}

Lastly we present two $\mathcal{O}(N^2\log(N))$ algorithms that give a lower and upper bound respectively for the optimal score. While the born-free matching algorithm does always manage to win, it doesn't perform as well as many known heuristic algorithms in practice. For larger $N$ it usually manages to obtain about $56.89\%$ of the pot, see figure \ref{primeperf}. \\

The optimal score has been shown to be the maximum weight flat alternating cycle-free  matching on a particular graph. Thus we now get an upper bound for free, since this score can't ever exceed the unrestricted maximum weight matching on the same graph. Using the algorithm in \cite{fastmatch} this can be done in $\mathcal{O}(V(E + V\log(V)))$ time where $V$ is the number of vertices and $E$ is the number of edges, so this gives us an upper bound in $\mathcal{O}(N^2\log(N))$. This algorithm is very hard to implement so the implementation used for this paper is based on Edmond's algorithm with the slower time complexity of $\mathcal{O}(N^3)$ \cite{maxmatch}. This bound is very tight for the values of $N$ where the optimal score is known, see figure \ref{boundperf}. The values for the optimal score are taken from \cite{chess}. \\

Lastly there is the lower bound. This is achieved through a heuristic algorithm that starts with the maximum weight matching and tries to remove as little weight as possible to break all flat alternating cycles in the matching. Suppose we now orient each edge such that it goes from the lower rank to the upper, reversing the orientation for edges within our matching. This makes flat alternating cycles correspond to directed cycles in this new graph. Thus our problem of breaking all directed cycles is now a well known problem, the minimum feedback arc set problem. Using the heuristic algorithm in \cite{fas} this can be done in $\mathcal{O}(VE)$ time, meaning the total time complexity is still $\mathcal{O}(N^2\log(N))$. A comparison of the output of this algorithm to optimal scores can be seen in figure \ref{boundperf}. \\

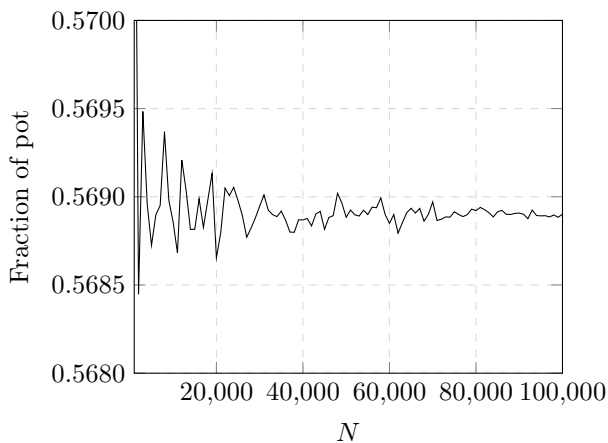
\begin{figure}[h!]
  \begin{center}
    \begin{tikzpicture}
      \begin{axis}[
          width=\textwidth*0.6,
          grid=major,
          grid style={dashed,gray!30}, 
          xlabel=$N$, 
          ylabel=Fraction of pot,
          xmin=1000, xmax=100000,
          ymin=0.568, ymax=0.57,   
          y tick label style={
          	/pgf/number format/.cd,
			fixed zerofill,
			precision=4,
			/tikz/.cd
		  },
		  x tick label style = {
		  	/pgf/number format/fixed,
		  	scaled x ticks = false,
		  	/tikz/.cd
		  },
		  ylabel near ticks
        ]
        \addplot[mark=none]
        table[x=N,y=p(N),col sep=comma] {tax_prime.csv}; 
      \end{axis}
    \end{tikzpicture}
    \caption{Performance of born-free matching algorithm.}
    \label{primeperf}
  \end{center}
\end{figure}

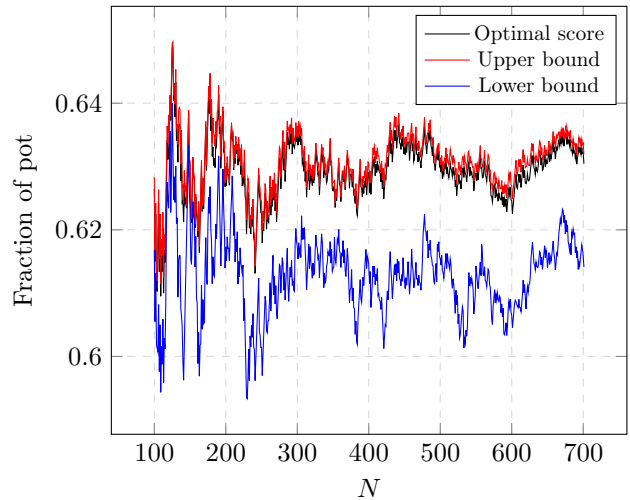
\begin{figure}[h!]
  \begin{center}
    \begin{tikzpicture}
      \begin{axis}[
          legend style={nodes={scale=0.8, transform shape}}, 
          grid=major, 
          grid style={dashed,gray!30},
          xlabel=$N$, 
          ylabel=Fraction of pot,          
        ]
        \addplot[mark=none, black]
        table[x=N,y=opt(N),col sep=comma] {tax_upper.csv}; 
        \addlegendentry {Optimal score};
        \addplot[mark=none, red]
        table[x=N,y=upper(N),col sep=comma] {tax_upper.csv}; 
        \addlegendentry {Upper bound};
        \addplot[mark=none, blue]
        table[x=N,y=lower(N),col sep=comma] {tax_lower.csv}; 
        \addlegendentry {Lower bound};
      \end{axis}
    \end{tikzpicture}
    \caption{Quality of matching upper and lower bounds.}
    \label{boundperf}
  \end{center}
\end{figure}

\end{document}